\documentclass{amsart}

\usepackage{amsfonts,amssymb,amsmath,amsthm,latexsym,booktabs,lscape,color,enumerate}

\theoremstyle{definition}
\newtheorem{definition}{Definition}[section]
\newtheorem{remark}[definition]{Remark}

\theoremstyle{plain}
\newtheorem{lemma}[definition]{Lemma}
\newtheorem{proposition}[definition]{Proposition}
\newtheorem{theorem}[definition]{Theorem}

\newtheorem{question}[definition]{Question}

\newcommand{\N}{\mathrm{N}}
\newcommand{\Ass}{\mathrm{Assoc}}
\newcommand{\End}{\mathrm{End}}
\newcommand{\Tder}{\mathrm{Tder}}
\newcommand{\QTder}{\mathrm{QTder}}

\begin{document}

\title{On the definitions of nucleus for dialgebras}

\author[S\'anchez-Ortega]{Juana S\'anchez-Ortega}

\address{Department of Algebra, Geometry and Topology, University of M\'alaga, Spain}

\email{jsanchezo@uma.es}

\subjclass[2010]{Primary 17D10. Secondary 17A20, 17A30, 17A32, 17D05.}

\keywords{nonassociative algebras and dialgebras, nucleus, ternary derivations, computer algebra.}

\maketitle

\begin{abstract}
Malcev dialgebras were introduced recently by Bremner, Peresi and S\'anchez-Ortega. In the present paper, we continue their study by introducing the notion of the generalized alternative di-nucleus of a 0-dialgebra. A general conjecture about the speciality of Malcev dialgebras in terms of this di-nucleus is formulated. In the last section, we introduce the appropriate generalization of the associative nucleus for dialgebras, and prove an analogue of 
Kleinfeld's theorem for the setting of dialgebras. 
\end{abstract}


\section{Introduction}

This paper is devoted to a better understanding of Malcev dialgebras, recently introduced by Bremner, Peresi and S\'anchez-Ortega \cite{BPSO}. 

It is well known that every associative algebra $A$ gives rise to a Lie algebra $A^-$, when the associative product $xy$ is replaced by the commutator $[x, y] = xy - yx$. And conversely, the famous Poincar\'e-Birkhoff-Witt Theorem \cite{J} states that any Lie algebra is isomorphic to a subalgebra of $A^-$ for some associative algebra $A$. 
Loday and Pirashvili \cite{LP} proved that the Poincar\'e-Birkhoff-Witt Theorem remains true for Leibniz algebras, a ``noncommutative'' version of Lie algebras; see also Aymon and Grivel \cite{APP}, Insua and Ladra \cite{IL} for other approaches. The role played by associative algebras is taken now by associative dialgebras, introduced by Loday \cite{Loday2, Loday3} in the last decade of the 20 th century. 

For nonassociative algebras, Malcev \cite{Malcev} showed that the commutator in an alternative algebra satisfies the defining identities for Malcev algebras. This result has been extended to the setting of dialgebras (see \cite[Section 4]{BPSO} for details). However, nowadays it still remains an open problem whether any Malcev algebra is isomorphic to a subalgebra of $A^-$ for some alternative algebra $A$. Some partial results on this problem were obtained in \cite{F, Sv}. Therefore, trying to solve the analogous problem in the dialgebra setting seems to be a very ambitious task. 

The speciality of Malcev algebras has been also treated from another perspective: a nice result of P\'erez-Izquierdo and Shestakov \cite{PIS} establishes that any Malcev algebra can be naturally embedded into a subalgebra of the generalized alternative nucleus for some nonassociative algebra. The generalized alternative nucleus, previously introduced by Morandi and P\'erez-Izquierdo \cite{MPI} in the context of composition algebras, turns out to be a Malcev algebra with the commutator. 

The purpose of the present paper is to develop the necessary machinery to approach similar problems for Malcev dialgebras in future projects. The paper is organized as follows: in section 2 we gather together basic definitions from the theory of associative and nonassociative dialgebras, and present a simplified statement of the general Kolesnikov-Pozhidaev (KP) algorithm for converting an arbitrary variety of multioperator algebras into a variety of dialgebras. The relationship between the KP algorithm and the Bremner-S\'anchez-Ortega (BSO) algorithm (for extending multilinear operations in an associative algebra to an associative dialgebra) is also described. In section 3 we apply the KP algorithm to the definining identities for the generalized alternative nucleus; as a result, we obtain a system of polynomial identities which will define the so-called generalized alternative di-nucleus. The rest of the section focusses on the construction of a Malcev dialgebra from the generalized alternative di-nucleus. Section 4 begins with some observations about the necessity to have a nonlinear di-Malcev identity similar to the identity for Malcev algebras expressed in terms of the Jacobian. After applying the BSO to the Jacobian we find the natural candidate to be called the nonlinear di-Malcev identity, but unfortunately it turns out not to be equivalent to the di-Malcev identity. Therefore, it is natural to ask about the existence of a nonlinear di-Malcev identity; we close section 4 with this question. In Section 5 we formulate a general conjecture about the speciality of Malcev dialgebras in terms of the generalized alternative di-nucleus. Finally in the last section, inspired by a classical result of Kleinfeld \cite{Kl} which measures the associativeness of a semiprime algebra by imposing some conditions on the associators and the associative nucleus; we introduce the proper generalization of the associative nucleus to the setting of dialgebras and study whether Kleinfeld's theorem remains true.


\section{Preliminaries}
  
\subsection{Associative and alternative dialgebras. Leibniz algebras}

Associative dialgebras were introduced by Loday \cite{Loday2, Loday3}
to provide a natural setting for Leibniz algebras.

The concept of Leibniz algebra was originally introduced in the mid-1960's by Bloh 
under the name ``D-algebra''.
 
\begin{definition} (Bloh \cite{Bloh1,Bloh2}, Loday \cite{Loday1}, Cuvier \cite{Cuvier}.)
A \textbf{(right) Leibniz algebra} is a vector space $L$, with a bilinear map 
$L \times L \to L$, denoted $(x,y) \mapsto \langle x,y \rangle$, satisfying the 
\textbf{(right) Leibniz identity}, which states that right multiplications are derivations:
  \begin{equation} \label{Rightleibnizidentity} \tag{L}
  \langle \langle x, y \rangle, z\rangle 
  \equiv 
  \langle \langle x, z \rangle, y \rangle 
  + 
  \langle x, \langle y, z \rangle \rangle.
  \end{equation}
If $\langle x, x \rangle \equiv 0$ then the Leibniz identity is the Jacobi identity and $L$
is a Lie algebra.
\end{definition}

An associative algebra becomes a Lie algebra if the product $xy$ is replaced by the Lie bracket $xy - yx$. 
The notion of dialgebra gives, by a similar procedure, a Leibniz algebra. Loday's idea was to replace the associative products $xy$ and $yx$ by two distinct operations, so that the resulting bracket is not necessarily skew-symmetric.

\begin{definition} 
\label{definitiondialgebras}
(Loday \cite{Loday2})
A \textbf{dialgebra} is a vector space $D$ with two bilinear operations
$\dashv\colon D \times D \to D$ and $\vdash\colon D \times D \to D$,
called the \textbf{left} and \textbf{right products}.
\end{definition} 

\begin{definition} 
\label{definition0dialgebras}
(Kolesnikov \cite{K})
A \textbf{0-dialgebra} is a dialgebra satisfying the \textbf{left} and \textbf{right bar identities}: 
  \[ 
  ( x \dashv y ) \vdash z
  \equiv
  ( x \vdash y ) \vdash z,
  \qquad
  x \dashv ( y \dashv z )
  \equiv
  x \dashv ( y \vdash z ).
  \]
Bar identities say that on the bar side of the products, 
the operation symbols are interchangeable. 
\end{definition} 

\begin{definition} 
\label{definitionassociativedialgebras}
(Loday \cite{Loday2})
An \textbf{associative dialgebra} is a 0-dialgebra satisfying \textbf{left, right} and \textbf{inner associativity}:
  \[
  ( x \dashv y ) \dashv z
  \equiv
  x \dashv ( y \dashv z ),
  \quad
  ( x \vdash y ) \vdash z
  \equiv
  x \vdash ( y \vdash z ),
  \quad
  ( x \vdash y ) \dashv z
  \equiv
  x \vdash ( y \dashv z ).
  \]
\end{definition} 

\begin{definition}
In any dialgebra $D$ the \textbf{dicommutator} is the bilinear operation 
  \[
  \langle x, y \rangle = x \dashv y - y \vdash x.
  \]
In what follows, we will write $D^-$ to denote $(D, \langle -, - \rangle)$, i.e., the underlying vector space of $D$ with the dicommutator. 
\end{definition}
It is easy to check that the dicommutator in any associative dialgebra $D$ satisfies the Leibniz identity, and hence $D$ 
gives rise to a Leibniz algebra $D^-$. Conversely, to motivate the definition of an associative dialgebra, suppose we are 
given a vector space $D$ with bilinear maps $\dashv$ and $\vdash$, and we want to determine the identities 
that must be satisfied so that the dicommutator satisfies the Leibniz identity. We calculate as follows:
  \allowdisplaybreaks
  \begin{align*}
  &
  \langle \langle x, y \rangle ,z \rangle - \langle \langle x, z \rangle, y \rangle - \langle x, \langle y, z \rangle 
  \rangle
  =
  \\
  &
  \big( ( x {\,\dashv\,} y ) {\,\dashv\,} z - x {\,\dashv\,} ( y {\,\dashv\,} z ) \big)
  - \big( ( x {\,\dashv\,} z ) {\,\dashv\,} y - x {\,\dashv\,} ( z {\,\vdash\,} y ) \big) 
  - \big( ( y {\,\vdash\,} x ) {\,\dashv\,} z - y {\,\vdash\,} ( x {\,\dashv\,} z ) \big)
  \\
  &
  - \big( y {\,\vdash\,} ( z {\,\vdash\,} x ) - ( y {\,\dashv\,} z ) {\,\vdash\,} x \big)
  - \big( z {\,\vdash\,} ( x {\,\dashv\,} y ) - ( z {\,\vdash\,} x ) {\,\dashv\,} y \big) 
  + \big( z {\,\vdash\,} ( y {\,\vdash\,} x ) - ( z {\,\vdash\,} y ) {\,\vdash\,} x \big).
  \end{align*}
If we set the differences within each pair of large parentheses to zero, we obtain identities equivalent to 
the defining identities for associative dialgebras. 

Ten years after Loday's definition of associative dialgebras, Liu \cite{Liu} introduced 
alternative dialgebras, the natural analogue of alternative algebras in the setting of structures with two operations

\begin{definition}
\label{alternativedialgebra}
(Liu \cite{Liu})
An \textbf{alternative dialgebra} is a 0-dialgebra satisfying
  \[
  ( x, y, z )_\dashv + ( z, y, x )_\vdash \equiv 0,
  \quad
  ( x, y, z )_\dashv - ( y, z, x )_\vdash \equiv 0,
  \quad
  ( x, y, z )_\times + ( x, z, y )_\vdash \equiv 0,  
  \]
where 
  \allowdisplaybreaks
  \begin{align*}
  ( x, y, z )_\dashv &= ( x \dashv y ) \dashv z - x \dashv ( y \dashv z ),
  \\
  ( x, y, z )_\times &= ( x \vdash y ) \dashv z - x \vdash ( y \dashv z ),
  \\
  ( x, y, z )_\vdash &= ( x \vdash y ) \vdash z - x \vdash ( y \vdash z ),  
  \end{align*}
are the {\bf left}, {\bf inner} and {\bf right associators}, respectively. We will refer to them as the {\bf dialgebra associators}. It is straightforward to check that every associative dialgebra is an alternative dialgebra.
\end{definition}

\subsection{KP algorithm}

Kolesnikov \cite{K} introduced a general algorithm for transforming the defining polynomial identities for a variety of binary algebras into the defining identities for the corresponding variety of dialgebras. This procedure was extended by Pozhidaev \cite{P2} to varieties of arbitrary $n$-ary algebras. In this subsection, we recall a simplified statement of the Kolesnikov-Pozhidaev (KP) algorithm given in \cite{BFSO}. See Chapoton \cite{Chapoton}, Vallette \cite{Vallette}, Kolesnikov and Voronin \cite{KolesnikovVoronin} for the underlying construction of the KP algorithm in the theory of operads.
  
Consider a multilinear $n$-ary operation $\{-,\dots,-\}$, and introduce $n$ new $n$-ary operations 
$\{-,\dots,-\}_j$ distinguished by subscripts $j = 1, \dots, n$.

First, we introduce the following \textbf{0-identities} for $i, j = 1, \dots, n$ with $i \ne j$ 
and $k, \ell = 1, \dots, n$; these identities say that the new operations are interchangeable 
in argument $i$ of operation $j$ when $i \ne j$:
  \allowdisplaybreaks
  \begin{align*}
  &
  \{ a_1, \dots, a_{i-1}, \{ b_1, \cdots, b_n \}_k, a_{i+1}, \dots, a_n \}_j
  \equiv
  \\
  &
  \{ a_1, \dots, a_{i-1}, \{ b_1, \cdots, b_n \}_\ell, a_{i+1}, \dots, a_n \}_j.
  \end{align*}

Note that the 0-identities are generalizations of the bar identities for associative dialgebras.  
  
Second, we consider a multilinear identity $I(a_1,\dots,a_d)$ of degree $d$ in the $n$-ary operation $\{-,\dots,-\}$.
We apply the following rule to each monomial of the identity; let $a_{\pi(1)} a_{\pi(2)} \dots a_{\pi(d)}$ 
be such a monomial with some placement of operation symbols where $\pi$ is a permutation of $1,\dots,d$.
For $i = 1, \dots, d$ we convert this monomial into a new monomial of the same degree in the $n$ new operations 
according to the position of the variable $a_i$, called the central argument of the monomial. For each occurrence of the original operation, we have the following cases:
  \begin{itemize}
  \item
  If $a_i$ occurs in argument $j$ then $\{\dots\}$ becomes $\{\dots\}_j$.
  \item
  If $a_i$ does not occur in any argument then
    \begin{itemize}
    \item
    if $a_i$ occurs to the left of the original operation, $\{\dots\}$ becomes $\{\dots\}_1$,
    \item
    if $a_i$ occurs to the right of the original operation, $\{\dots\}$ becomes $\{\dots\}_n$.
    \end{itemize}
  \end{itemize}
The resulting new identity is called the \textbf{KP identity} corresponding to 
$I(a_1,\dots,a_d)$.

The choice of new operations, in the two subcases under the second bullet above,
gives a convenient normal form for the monomial: by the 0-identities, the subscripts 1 and $n$ 
can be replaced by any other subscripts.  Suppose that $a_i$ is the central
argument and that the identity $I(a_1,\dots,a_d)$ contains a monomial of this form:
  \[
  \{ \dots, 
  \overbrace{\{ -, \dots, - \}}^{\text{argument $i$}}, 
  \dots, 
  \overbrace{\{ \dots, a_i, \dots \}}^{\text{argument $j$}}, 
  \dots, 
  \overbrace{\{ -, \dots, - \}}^{\text{argument $k$}},
  \dots \}.
  \]
Since $a_i$ occurs in argument $j$, the outermost operation must receive subscript $j$:
  \[
  \{ \dots, \{ -, \dots, - \}, \dots, \{ \dots, a_i, \dots \}, \dots, \{ -, \dots, - \}, \dots \}_j.
  \]
Our convention above attaches subscripts $n$ and 1 to arguments $i$ and $k$ respectively:
  \[
  \{ \dots, \{ -, \dots, - \}_n, \dots, \{ \dots, a_i, \dots \}, \dots, \{ -, \dots, - \}_1, \dots \}_j.
  \]
Since these subscripts occur in arguments $i \ne j$ and $k \ne j$ of operation $j$, 
the 0-identities imply that any other subscripts would give an equivalent identity.

\begin{remark}
Applying the KP algorithm to the associativity law $\{\{x, y \},z\} \equiv \{x,\{y,z\}\}$ gives 
the defining identities for an associative dialgebra. The defining identities for an alternative dialgebra can be obtained by an application to the KP algorithm to the linearization of right and left alternativity: $(x, x, y) \equiv 0$ and $(x, y, y) \equiv 0$. (See \cite[Examples 7 and 8]{BPSO})
\end{remark}

\subsection{BSO algorithm}

Bremner and the author \cite{BSO} have introduced an algorithm (BSO) for extending multilinear operations in an associative algebra to corresponding operations in an associative dialgebra. The BSO algorithm is based on the following notion.

\begin{definition} (Loday \cite{Loday3})
A {\bf dialgebra monomial} in the free 0-dialgebra on a set $X$ of generators is a product $x = \overline{x_1 x_2 \cdots x_n}$ 
where $x_1, \ldots, x_n \in X$ and the bar indicates some placement of parentheses and some choice of 
operations. 
The {\bf center} of $x$ is defined inductively: if $n = 1$ ($x \in X$) then $c(x) = x$; if $n \ge 2$ then
$x = y \dashv z$ or $x = y \vdash z$ and $c( y \dashv z ) = c(y)$ or $c( y \vdash z ) = c(z)$
following the direction of the product symbols. Using other words, the center of a monomial is the element which has all the product symbols pointing inwards to it.
\end{definition}

See \cite[Definition 5.1]{BFSO} for a generalized statement of the BSO algorithm. The input is an $n$-multilinear operation $\omega(x_1, \ldots, x_n)$ in an associative algebra, and the output are $n$ operations in an associative dialgebra obtained by making $x_i$ the center of each monomial of $\omega$. For example, the Lie bracket $[x, y] = xy - yx$ gives rise to the operations: 
\[
[x, y]_1 = x \dashv y - y \vdash x, \qquad [x, y]_2 = x \vdash y - y \dashv x.
\]
Note that $[x,y]_2 = - [y,x]_1$, and moreover $[x, y]_1 = \langle x, y \rangle$, the dicommutator. 

The BSO algorithm also works for nonassociative algebras and dialgebras. We have already presented an example: the left, inner and right associators (see Definition \ref{alternativedialgebra}) can be obtained by applying the BSO algorithm to the associator in a nonassociative algebra.

\subsection{Relation between the KP and BSO algorithms}

In \cite[Section 6]{BFSO} a general conjecture was stated in terms of a commutative diagram relating the output of the KP and BSO algorithms. Given $\omega$, a multilinear $n$-ary operation in an associative algebra, this conjecture established that under a mild technical condition, the following two processes produce the same results:
 \begin{itemize}
  \item
  Find the identities satisfied by $\omega$, and apply the KP algorithm.
  \item
  Apply the BSO algorithm to $\omega$, and find the identities satisfied by $\omega_1,\dots,\omega_n$.
\end{itemize}
This conjecture has been recently proved by Kolesnikov and Voronin \cite{KolesnikovVoronin} using operads.

\subsection{Malcev dialgebras}
Malcev dialgebras, the appropriate generalization of Malcev algebras to the setting of dialgebras, have been recently introduced by Bremner, Peresi and the author \cite{BPSO}. Malcev dialgebras are related to alternative dialgebras in the same way that Malcev algebras are related to alternative algebras. 

Before stating their definition, let us first recall the definition of a Malcev algebra.

\begin{definition} 
(Malcev \cite{Malcev})
A \textbf{Malcev algebra} is a vector space with a bilinear operation $xy$ satisfying 
\textbf{anticommutativity} and the \textbf{Malcev identity}:
  \[
  x^2 \equiv 0,
  \qquad
  (xy)(xz) \equiv ((xy)z)x + ((yz)x)x + ((zx)x)y.
  \]
\end{definition}

\begin{lemma}
\label{malcevdefinition}
\emph{(Sagle \cite{Sagle})}
If the characteristic is not 2, then an algebra is Malcev if and only if 
it satisfies the following multilinear identities:
  \[ 
  xy + yx \equiv 0,
  \qquad
  (xz)(yt) \equiv ((xy)z)t + ((yz)t)x + ((zt)x)y + ((tx)y)z.
  \]
\end{lemma}

The defining identities for Malcev dialgebras were obtained by applying the KP algorithm to the multilinear identities displayed in the previous lemma. 

\begin{definition} (Bremner, Peresi, JSO \cite{BPSO}) \label{MDdefinition}
Over a field of characteristic not $2$, a \textbf{(right) Malcev dialgebra} 
is a vector space with a bilinear operation $xy$
satisfying \textbf{right anticommutativity} and the \textbf{di-Malcev identity}:
  \[
  x(yz) + x(zy) \equiv 0,
  \qquad
  ((xy)z)t - ((xt)y)z - (x(zt))y - (xz)(yt) - x((yz)t) \equiv 0.
  \]
\end{definition}

\section{The generalized alternative di-nucleus}

Malcev \cite{Malcev} showed that an alternative algebra becomes a Malcev algebra 
by considering the same underlying vector space under the commutator. Bremner, Peresi and the author \cite{BPSO} have used computer algebra to show that any subspace of an alternative dialgebra closed under the dicommutator is a Malcev dialgebra. 
 
A few years ago, in 2004, P\'erez-Izquierdo and Shestakov \cite{PIS} established a more general way of constructing Malcev algebras. Given an algebra $A$, the {\bf generalized alternative nucleus} $\mathrm{N_{alt}}(A)$ of $A$, introduced by Morandi and P\'erez-Izquierdo \cite{MPI}, is defined as 
\[
\mathrm{N_{alt}}(A) = \{a \in A \mid (a, x, y) = -(x, a, y) = (x, y, a) \, \, \mbox{ for all }  \, \, x, y \in A\},
\]
where $(x, y, z) = (xy)z - x(yz)$ denotes the associator of $A$. As was pointed out in \cite{MPI},
$\mathrm{N_{alt}}(A)$ may not be a subalgebra of $A$ but it is closed under the commutator, so it is a
subalgebra of $A^-$, that is, $(A, [-,-])$. Moreover, $\mathrm{N_{alt}}(A)^-$ is a Malcev algebra. (See \cite[Proposition 4.3]{MPI}) 

\begin{remark}
It is easy to see that the elements of $\mathrm{N_{alt}}(A)$ satisfy right and left alternativity, i.e., the defining identities for an alternative algebra. In fact, if A is an alternative algebra then $\mathrm{N_{alt}}(A) = A$, and the construction of Malcev algebras from alternative algebras is recovered. 
\end{remark}

In this section we introduce the analogue of the generalized alternative nucleus for dialgebras. We prove that it is closed under the dicommutator and satisfies the defining identities for Malcev dialgebras.

\subsection{Definition of the generalized alternative di-nucleus of a 0-dialgebra} \label{Nalt}

In this subsection we apply the KP algorithm to the defining identities for the generalized alternative nucleus $\mathrm{N_{alt}}(A)$ of an algebra $A$. Expanding the associators and using the operation symbol $\{-, -\}$ we get the following identities:
\allowdisplaybreaks
\begin{align*}
\{ \{ a, x \}, y \}Ê+ \{ \{ x, a \}, y \}Ê- \{ a, \{ x, y \} \}Ê- \{ x, \{ a, y \} \}Ê& \equiv 0,
\\
\{ \{ x, y \}, a \}Ê+ \{ \{ x, a \}, y \}Ê- \{ x, \{ y, a \} \}Ê- \{ x, \{ a, y \} \}Ê& \equiv 0.
\end{align*}
The KP identities are obtained by making $a$, $x$, $y$ in turn the central argument:
\allowdisplaybreaks
\begin{align*}
\{ \{ a, x \}_1, y \}_1Ê+ \{ \{ x, a \}_2, y \}_1Ê- \{ a, \{ x, y \}_1 \}_1Ê- \{ x, \{ a, y \}_1 \}_2Ê& \equiv 0,
\\
\{ \{ a, x \}_2, y \}_1Ê+ \{ \{ x, a \}_1, y \}_1Ê- \{ a, \{ x, y \}_1 \}_2Ê- \{ x, \{ a, y \}_1 \}_1Ê& \equiv 0,
\\
\{ \{ a, x \}_2, y \}_2Ê+ \{ \{ x, a \}_2, y \}_2Ê- \{ a, \{ x, y \}_2 \}_2Ê- \{ x, \{ a, y \}_2 \}_2Ê& \equiv 0,
\\
\{ \{ x, y \}_2, a \}_2Ê+ \{ \{ x, a \}_2, y \}_1Ê- \{ x, \{ y, a \}_2 \}_2Ê- \{ x, \{ a, y \}_1 \}_2Ê& \equiv 0,
\\
\{ \{ x, y \}_1, a \}_1Ê+ \{ \{ x, a \}_1, y \}_1Ê- \{ x, \{ y, a \}_1 \}_1Ê- \{ x, \{ a, y \}_1 \}_1Ê& \equiv 0,
\\
\{ \{ x, y \}_2, a \}_1Ê+ \{ \{ x, a \}_2, y \}_2Ê- \{ x, \{ y, a \}_1 \}_2Ê- \{ x, \{ a, y \}_2 \}_2Ê& \equiv 0,
\end{align*}
Using the notation $\, \star \dashv \bullet = \{\star, \bullet \}_1$ and $\star \vdash \bullet = \{\star, \bullet \}_2$, the identities above become
\allowdisplaybreaks 
\begin{align}
(a \dashv x) \dashv  y Ê+ (x \vdash a) \dashv y Ê- a \dashv (x \dashv y) - x \vdash (a \dashv y) & \equiv 0, \label{1}
\\
(a \vdash x) \dashv  y Ê+ (x \dashv a) \dashv y Ê- a \vdash (x \dashv y) - x \dashv (a \dashv y) & \equiv 0,
\\
(a \vdash x) \vdash  y Ê+ (x \vdash a) \vdash y Ê- a \vdash (x \vdash y) - x \vdash (a \vdash y) & \equiv 0,
\\
(x \vdash y) \vdash  a Ê+ (x \vdash a) \dashv y Ê- x \vdash (y \vdash a)Ê- x \vdash (a \dashv y)Ê& \equiv 0,
\\
(x \dashv y) \dashv  a Ê+ (x \dashv a) \dashv y Ê- x \dashv (y \dashv a)Ê- x \dashv (a \dashv y)Ê& \equiv 0,
\\
(x \vdash y) \dashv  a Ê+ (x \vdash a) \vdash y Ê- x \vdash (y \dashv a)Ê- x \vdash (a \vdash y)Ê& \equiv 0.\label{6}
\end{align}
Using the dialgebra associators, identities \eqref{1}-\eqref{6} can be rewritten as follows:
\allowdisplaybreaks
\begin{align*}
(a, x, y)_\dashv + (x, a, y)_\times & \equiv 0,
\\
(a, x, y)_\times + (x, a, y)_\dashv & \equiv 0,
\\
(a, x, y)_\vdash + (x, a, y)_\vdash & \equiv 0,
\\
(x, y, a)_\vdash + (x, a, y)_\times & \equiv 0,
\\
(x, y, a)_\dashv + (x, a, y)_\dashv & \equiv 0,
\\
(x, y, a)_\times + (x, a, y)_\vdash & \equiv 0.
\end{align*}
To finish, note that the 0-identities
\[
\{ a, \{ b , c \}_1 \}_1Ê\equiv \{ a, \{ b , c \}_2 \}_1, \qquad 
\{ \{ a, b \}_1, c \}_2Ê\equiv \{ \{ a, b \}_2, c \}_2,
\]
become the bar identities by replacing the symbols $\{ - , - \}_1$, $\{ - , - \}_2$ by 
 $\dashv$ and $\vdash$, respectively. More precisely,
\[
a \dashv (b \dashv c) Ê\equiv a \dashv (b \vdash c) , \qquad 
(a \dashv b) \vdash c  \equiv (a \vdash b) \vdash c.
\]
The above calculations make possible to introduce the generalized alternative nucleus for the setting of dialgebras.
\begin{definition}
The {\bf generalized alternative di-nucleus} $\mathrm{N_{alt}}(D)$ of a 0-dialgebra $D$ is the set of elements $a\in D$ which satisfies the bar identities and the following:
\allowdisplaybreaks
\begin{align}
(a, x, y)_\dashv & \equiv -(x, a, y)_\times \equiv (x, y, a)_\vdash, 
\tag{GAN1} \label{GAN1}
\\
(a, x, y)_\times & \equiv -(x, a, y)_\dashv \equiv (x, y, a)_\dashv,
\tag{GAN2} \label{GAN2}
\\
(a, x, y)_\vdash & \equiv -(x, a, y)_\vdash \equiv (x, y, a)_\times.
\tag{GAN3} \label{GAN3}
\end{align}
\end{definition}

\begin{proposition} \label{nalt}
Let $D$ be a 0-dialgebra. Then the elements of $\mathrm{N_{alt}}(D)$ satisfy the defining identities for an alternative dialgebra. Moreover, if $D$ is alternative then $\mathrm{N_{alt}}(D) = D$. 
\end{proposition}

\begin{proof}
Straightforward.
\end{proof}

\begin{remark}
In the setting of algebras, the generalized alternative nucleus does not have a subalgebra structure. As a consequence, we can affirm that $\mathrm{N_{alt}}(D)$ is not in general a sub-dialgebra of $D$.
\end{remark}
 
Our next goal will be to show that $\mathrm{N_{alt}}(D)$ is closed under the dicommutator. Our argument shares some of the ideas developed in \cite[Section 4]{MPI}. We start by introducing some definitions. 

\begin{definition}
Let $D$ be a dialgebra and $a \in D$. The {\bf multiplication operators} $L^{\vdash}_a, L^{\dashv}_a, R^{\vdash}_a, R^{\dashv}_a: D \to D$ are given by 
\[
L^{\vdash}_a (x)  = a \vdash x, \quad 
L^{\dashv}_a (x)  = a \dashv x, \quad
R^{\vdash}_a (x)  = x \vdash a, \quad
R^{\dashv}_a (x)  = x \dashv a, 
\]
for any $x\in D$. We also introduce $\, T^{\times}_a := L^{\vdash}_a + R^{\dashv}_a, \quad \widetilde{T}^{\times}_a := R^{\vdash}_a + L^{\dashv}_a.$
\end{definition}

\begin{definition}
A {\bf ternary derivation} of a 0-dialgebra $D$ is a triple $(\delta_1, \delta_2, \delta_3) \in \End_F(D) \times \End_F(D) \times \End_F(D)$ such that 
\begin{equation} \label{ternary}
\delta_1(x \dashv y) = \delta_2(x) \dashv y + x \dashv \delta_3(y), 
\qquad
\delta_1(x \vdash y) = \delta_2(x) \vdash y + x \vdash \delta_3(y),
\end{equation}
for all $x, \, y \in D$. 
\end{definition}

The set of all ternary derivations $\Tder(D)$ of $D$ has a Lie algebra structure with the Lie bracket defined to be
\begin{equation} \label{bracket}
[(\delta_1, \delta_2, \delta_3), \, (\mu_1, \mu_2, \mu_3)]:= ([\delta_1, \mu_1], [\delta_2, \mu_2], [\delta_3, \mu_3]),
\end{equation}
for all $(\delta_1, \delta_2, \delta_3)$, $(\mu_1, \mu_2, \mu_3) \in \Tder(D)$.
In case $\delta_1 = \delta_2 = \delta_3$, equation \eqref{ternary} says that $\delta_1$ is a derivation of $D$. 

\begin{remark}
Using the terminology of ternary derivations, it is easy to see that an element $a \in D$ satisfies (\ref{GAN2}) and (\ref{GAN3}) if and only if $(L^{\vdash}_a, \, T^{\times}_a, -L^{\vdash}_a)$, $(R^{\dashv}_a, \, -R^{\dashv}_a, T^{\times}_a) \in \Tder(D)$.
 
At this point, one may ask about the remaining condition (\ref{GAN1}). It is not surprising that (\ref{GAN1}) will be related to the operators $L^{\dashv}_a$, $R^{\vdash}_a$, $\widetilde{T}^{\times}_a$.

Unfortunately, the triples $(L^{\dashv}_a, \, \widetilde{T}^{\times}_a, -L^{\dashv}_a)$, $(R^{\vdash}_a, \, -R^{\vdash}_a, \widetilde{T}^{\times}_a)$ are no longer ternary derivations of $D$. To be more precise, given $a \in D$, it follows that
\begin{align*}
(x, a, y)_\times & \equiv -(a, x, y)_\vdash \Leftrightarrow L^{\dashv}_a (x\dashv y) = \widetilde{T}^{\times}_a(x) \dashv y - x \vdash L^{\dashv}_a (y), 
\\
(x, a, y)_\times & \equiv -(x, y, a)_\vdash \Leftrightarrow R^{\vdash}_a (x\vdash y) = - R^{\vdash}_a (x) \dashv y + x \vdash \widetilde{T}^{\times}_a(y),  
\end{align*}
for any $x, y \in D$. Note that it is natural to obtain the above expressions since the bar identities can be applied to get 
\[
L^{\dashv}_a (x\dashv y) = L^{\dashv}_a (x\vdash y), \qquad R^{\vdash}_a (x\vdash y) = R^{\vdash}_a (x\dashv y),
\]
for all $x, y \in D$. 
\end{remark}

\begin{definition}
A triple $(\delta_1, \delta_2, \delta_3) \in \End_F(D) \times \End_F(D) \times \End_F(D)$ of a 0-dialgebra $D$, is called a {\bf quasi-ternary derivation} if it satisfies 
\[
\delta_1(x \dashv y) = \delta_2(x) \dashv y + x \vdash \delta_3(y),
\qquad \delta_1(x \vdash y) = \delta_2(x) \dashv y + x \vdash \delta_3(y),
\]
for all $x, y \in D$. Denote by $\QTder(D)$ the set of all quasi-ternary derivations of $D$.

\end{definition}

Note that an element $a \in D$ satisfies (\ref{GAN1}) if and only if $(L^{\dashv}_a, \, \widetilde{T}^{\times}_a, -L^{\dashv}_a)$, $(R^{\vdash}_a, \, -R^{\vdash}_a, \widetilde{T}^{\times}_a)$ are quasi-ternary derivations. We have just proved the following result.

\begin{lemma} \label{caract}
Let $D$ be 0-dialgebra and $a \in D$. Then $a\in \mathrm{N_{alt}}(D)$ if and only if $a$ satisfies the following conditions:
\begin{enumerate}[(i)]\itemsep=2mm
\item[\rm(i)] $(L^{\vdash}_a, \, T^{\times}_a, -L^{\vdash}_a)$, $\, (R^{\dashv}_a, \, -R^{\dashv}_a, T^{\times}_a) \in \Tder(D)$.
\item[\rm(ii)] $(L^{\dashv}_a, \, \widetilde{T}^{\times}_a, -L^{\dashv}_a)$, $(R^{\vdash}_a, \, -R^{\vdash}_a, \widetilde{T}^{\times}_a) \in \QTder(D)$.
\end{enumerate}
\end{lemma}

The following lemma relates the product of a ternary derivation with a quasi-ternary derivation.

\begin{lemma} \label{deri}
Let $D$ be a 0-dialgebra. If $(\delta_1, \delta_2, \delta_3)\in \QTder(D)$ and $(\mu_1, \mu_2, \mu_3)\in \Tder(D)$, then 
$([\delta_1, \mu_1], [\delta_2, \mu_2], [\delta_3, \mu_3]) \in \QTder(D)$.
\end{lemma}

\begin{proof}
Given $x, y \in D$, we have
\begin{align*}
[\delta_1, \mu_1](x \dashv y) & = \delta_1 (\mu_1(x\dashv y)) - \mu_1 (\delta_1(x\dashv y)) = \delta_1(\mu_2(x) \dashv y) + \delta_1(x \dashv \mu_3(y)) 
\\
& - \mu_1(\delta_2(x) \dashv y) - \mu_1(x \vdash \delta_3(y)) = \delta_2(\mu_2(x)) \dashv y + \mu_2(x) \vdash \delta_3(y)
\\
& + \delta_2(x) \dashv \mu_3(y) + x \vdash \delta_3(\mu_3(y)) - \mu_2(\delta_2(x)) \dashv y - \delta_2(x) \dashv \mu_3(y)
\\
& - \mu_2(x) \vdash \delta_3(y) - x \vdash \mu_3(\delta_3(y)) = [\delta_2,\mu_2](x) \dashv y + x \dashv [\delta_3,\mu_3](y),
\end{align*}
as desired.
\end{proof}

Next, we derive some properties of the multiplication operators that will be useful for our purposes.

\begin{lemma} \label{properties}
Let $D$ be a 0-dialgebra, $a, b \in \mathrm{N_{alt}}(D)$ and $x \in D$. Then 
\begin{enumerate}[(i)]\itemsep=2mm
\item[\rm(i)] $L^{\vdash}_{ a \dashv x } = L^{\vdash}_a L^{\vdash}_x + [ R^{\dashv}_a, \, L^{\vdash}_x]$, 
        $\qquad L^{\vdash}_{ x \vdash a } = L^{\vdash}_x L^{\vdash}_a + [ L^{\vdash}_x, \, R^{\dashv}_a]$.
\item[\rm(ii)] $L^{\dashv}_{ a \dashv x } = L^{\dashv}_a L^{\dashv}_x + [ R^{\vdash}_a, \, L^{\vdash}_x]$, 
          $\qquad L^{\dashv}_{ x \vdash a } = L^{\vdash}_x L^{\dashv}_a + [ L^{\vdash}_x, \, R^{\vdash}_a]$.
\item[\rm(iii)] $R^{\vdash}_{ a \dashv x } = R^{\dashv}_x R^{\vdash}_a + [ R^{\dashv}_x, \, L^{\dashv}_a]$, 
          $\qquad R^{\vdash}_{ x \vdash a } = R^{\vdash}_a R^{\vdash}_x + [ L^{\dashv}_a, \, R^{\dashv}_x]$.
\item[\rm(iv)] $R^{\dashv}_{ a \dashv x } = R^{\dashv}_x R^{\dashv}_a + [ R^{\dashv}_x, \, L^{\vdash}_a]$, 
          $\qquad R^{\dashv}_{ x \vdash a } = R^{\dashv}_a R^{\dashv}_x + [ L^{\vdash}_a, \, R^{\dashv}_x]$.
\item[\rm(v)] $[L^{\vdash}_a, R^{\dashv}_b]Ê= [R^{\dashv}_a, L^{\vdash}_b]$, 
       $\qquad [L^{\dashv}_a, R^{\dashv}_b]Ê= [R^{\vdash}_a, L^{\vdash}_b]$.
\item[\rm(vi)] $L^{\vdash}_{\langle a, b \rangle} = [L^{\vdash}_a, L^{\vdash}_b] + 2[R^{\dashv}_a, L^{\vdash}_b]$,
        $\qquad L^{\dashv}_{\langle a, b \rangle} = [L^{\dashv}_a, L^{\vdash}_b] + 2[R^{\vdash}_a, L^{\vdash}_b]$.
\item[\rm(vii)] $-L^{\vdash}_{\langle a, b \rangle} = [L^{\vdash}_a, L^{\vdash}_b] - 2[T^{\times}_a, L^{\vdash}_b]$,
        $\quad -L^{\dashv}_{\langle a, b \rangle} = [L^{\dashv}_a, L^{\vdash}_b] - 2[\widetilde{T}^{\times}_a, L^{\vdash}_b]$.          
\item[\rm(viii)] $R^{\dashv}_{\langle a, b \rangle} = -[R^{\dashv}_a, R^{\dashv}_b] - 2[L^{\vdash}_a, R^{\dashv}_b]$,
         $\qquad R^{\vdash}_{\langle a, b \rangle} = -[R^{\vdash}_a, R^{\dashv}_b] - 2[L^{\dashv}_a, R^{\dashv}_b]$.
\item[\rm(ix)] $-R^{\dashv}_{\langle a, b \rangle} = -[R^{\dashv}_a, R^{\dashv}_b] + 2[T^{\times}_a, R^{\dashv}_b]$,
         $\quad - R^{\vdash}_{\langle a, b \rangle} = -[R^{\vdash}_a, R^{\dashv}_b] + 2[\widetilde{T}^{\times}_a, R^{\dashv}_b]$.
\item[\rm(x)] $T^{\times}_{\langle a, b \rangle} = [T^{\times}_a, T^{\times}_b] - 2[R^{\dashv}_a, T^{\times}_b] = -[T^{\times}_a, T^{\times}_b] + 2[L^{\vdash}_a, T^{\times}_b]$.
\item[\rm(xi)]
$\widetilde{T}^{\times}_{\langle a, b \rangle} = [\widetilde{T}^{\times}_a, T^{\times}_b] - 2[R^{\vdash}_a, T^{\times}_b]=
-[\widetilde{T}^{\times}_a, T^{\times}_b] + 2[L^{\dashv}_a, T^{\times}_b]$. 
\end{enumerate}
\end{lemma}

\begin{proof}
Let $a, b \in \mathrm{N_{alt}}(D)$ and $x, y \in D$.

(i). Applying the left bar identity we get
\[
L^{\vdash}_{a \dashv x} (y) = (a \dashv x) \vdash y = (a \vdash x) \vdash y.
\]
On the other hand, we have
\[
L^{\vdash}_a L^{\vdash}_x (y) + [ R^{\dashv}_a, \, L^{\vdash}_x](y) = 
a \vdash (x \vdash y) + (x \vdash y) \dashv a - x \vdash (y \dashv a).
\]
Thus, $L^{\vdash}_{a \dashv x} (y) = L^{\vdash}_a L^{\vdash}_x (y) + [ R^{\dashv}_a, \, L^{\vdash}_x](y)$ if and only if $(a, x, y)_\vdash = - (x, y, a)_\times$ which holds since $a \in \mathrm{N_{alt}}(D)$. Analogously, one can show that $L^{\vdash}_{x \vdash a}(y) = L^{\vdash}_x L^{\vdash}_a (y) + [ L^{\vdash}_x, \, R^{\dashv}_a] (y)$.

(ii), (iii) and (iv) can be proved analogously.

(v). By definition we have
\begin{align*}
[L^{\vdash}_a, R^{\dashv}_b]Ê(y) & = a \vdash (y \dashv b) - (a\vdash y) \dashv b = -(a, y, b)_\times,
\\
[R^{\dashv}_a, L^{\vdash}_b] (y) & = b \vdash (y \dashv a) - (b \vdash y) \dashv a = (b, y, a)_\times.
\end{align*}
Applying that $b \in \mathrm{N_{alt}}(D)$ from (\ref{GAN2}) we obtain $(b, y, a)_\times = (y, a, b)_\dashv$. Since $a \in \mathrm{N_{alt}}(D)$ a second use of (\ref{GAN2}) gives $(y, a, b)_\dashv = -(a, y, b)_\times$. The proof of the second equality is similar: apply (\ref{GAN3}) with $b \in \mathrm{N_{alt}}(D)$, and (\ref{GAN1}) with $a \in \mathrm{N_{alt}}(D)$.

(vi) follows from (i), (ii) and (v).

(vii) follows from (vi) and the definitions of $T^{\times}_a$, $\widetilde{T}^{\times}_a$.

(viii) is a consequence of (iii), (iv) and (v).

(ix) is obtained by an application of (viii), taking into account the definitions of $T^{\times}_a$, $\widetilde{T}^{\times}_a$.

(x). Applying (vi), (viii) and (v) we get
\[
T^{\times}_{\langle a, b \rangle} = L^{\vdash}_{\langle a, b \rangle} + R^{\dashv}_{\langle a, b \rangle} = 
                                    [L^{\vdash}_a, L^{\vdash}_b] - [R^{\dashv}_a, R^{\dashv}_b].
\]
On the other hand
\begin{equation} \label{aux}
[T^{\times}_a, T^{\times}_b] = [L^{\vdash}_a + R^{\dashv}_a, L^{\vdash}_b + R^{\dashv}_b] = [L^{\vdash}_a, L^{\vdash}_b] + 2[R^{\dashv}_a, L^{\vdash}_b] + [R^{\dashv}_a, R^{\dashv}_b], 
\end{equation}
which implies 
\begin{align*}
T^{\times}_{\langle a, b \rangle} & = [T^{\times}_a, T^{\times}_b] - 2[R^{\dashv}_a, L^{\vdash}_b]
- 2[R^{\dashv}_a, R^{\dashv}_b] = [T^{\times}_a, T^{\times}_b] - 2[R^{\dashv}_a, T^{\times}_b]
\\
& = - [T^{\times}_a, T^{\times}_b] + 2[L^{\vdash}_a, T^{\times}_b].
\end{align*}

(xi) can be shown similarly. 
\end{proof}

\begin{proposition} \label{naltclosed}
The generalized alternative di-nucleus of a 0-dialgebra is closed under the dicommutator.
\end{proposition}

\begin{proof}
Let $D$ be a 0-dialgebra and $\mathrm{N_{alt}}(D)$ its generalized alternative di-nucleus. Given $a, b \in \mathrm{N_{alt}}(D)$; in order to show that $\langle a, b\rangle \in \mathrm{N_{alt}}(D)$ we are going to use the characterization in terms of ternary and quasi-ternary derivations described in Lemma \ref{caract}. Let us start by proving the claim that $(L^{\vdash}_{\langle a, b \rangle}, T^{\times}_{\langle a, b \rangle}, -L^{\vdash}_{\langle a, b \rangle}) \in \Tder(D)$.  
Since $a, b \in \mathrm{N_{alt}}(D)$, Lemma \ref{caract} (i) allows us to conclude that 
\[
(L^{\vdash}_a, \, T^{\times}_a, -L^{\vdash}_a), \quad 
(L^{\vdash}_b, \, T^{\times}_b, -L^{\vdash}_b), \quad
(R^{\dashv}_a, \, -R^{\dashv}_a, T^{\times}_a) \in \Tder(D), 
\]
which implies that
\[
 \left([L^{\vdash}_a, L^{\vdash}_b], \, [T^{\times}_a, T^{\times}_b], [L^{\vdash}_a,L^{\vdash}_b]\right) +
2\left([R^{\dashv}_a, L^{\vdash}_b], \, [-R^{\dashv}_a, T^{\times}_b], [T^{\times}_a,-L^{\vdash}_b]\right) \in \Tder(D).
\]
On the other hand, applying Lemma \ref{properties} (vi), (vii), (x) we get 
\begin{align*}
& \left([L^{\vdash}_a, L^{\vdash}_b], \, [T^{\times}_a, T^{\times}_b], [L^{\vdash}_a,L^{\vdash}_b]\right) +
2\left([R^{\dashv}_a, L^{\vdash}_b], \, [-R^{\dashv}_a, T^{\times}_b], [T^{\times}_a,-L^{\vdash}_b]\right) 
\\
& \quad = \left([L^{\vdash}_a, L^{\vdash}_b] + 2[R^{\dashv}_a, L^{\vdash}_b], \, [T^{\times}_a, T^{\times}_b] + 2[-R^{\dashv}_a, T^{\times}_b], \, [L^{\vdash}_a,L^{\vdash}_b] + 2[T^{\times}_a,-L^{\vdash}_b]\right) 
\\
& \quad = (L^{\vdash}_{\langle a, b \rangle}, T^{\times}_{\langle a, b \rangle}, -L^{\vdash}_{\langle a, b \rangle}),
\end{align*}
which concludes the proof of the claim. 

Similarly, one can show that $(R^{\dashv}_{\langle a, b \rangle}, -R^{\dashv}_{\langle a, b \rangle}, T^{\times}_{\langle a, b \rangle}) \in \Tder(D)$. It remains to check that $(L^{\dashv}_{\langle a, b \rangle}, \, \widetilde{T}^{\times}_{\langle a, b \rangle}, -L^{\dashv}_{\langle a, b \rangle})$, $(R^{\vdash}_{\langle a, b \rangle}, \, -R^{\vdash}_{\langle a, b \rangle}, \widetilde{T}^{\times}_{\langle a, b \rangle})$ $\in \QTder(D)$. Let us now prove that $(R^{\vdash}_{\langle a, b \rangle}, \, -R^{\vdash}_{\langle a, b \rangle}, \widetilde{T}^{\times}_{\langle a, b \rangle}) \in \QTder(D)$. Lemma \ref{caract} (ii) yields $(L^{\dashv}_a, \, \widetilde{T}^{\times}_a, -L^{\dashv}_a)$, $(R^{\vdash}_a, \, -R^{\vdash}_a, \widetilde{T}^{\times}_a) \in \QTder(D)$, while from Lemma \ref{caract} (i) we get that $(R^{\dashv}_b, \, -R^{\vdash}_b, T^{\times}_b) \in \Tder(D)$. Now apply Lemma \ref{deri} to conclude that 
\[
 - \left([R^{\vdash}_a, R^{\dashv}_b], \, [R^{\vdash}_a, R^{\dashv}_b], [\widetilde{T}^{\times}_a,T^{\times}_b]\right) 
 - 2\left([L^{\dashv}_a, R^{\dashv}_b], \, [\widetilde{T}^{\times}_a, -R^{\dashv}_b], [-L^{\dashv}_a,T^{\times}_b]\right) \in \QTder(D).
\]
On the other hand by Lemma \ref{properties} (viii), (ix), (xi), we obtain
\begin{align*}
& - \left([R^{\vdash}_a, R^{\dashv}_b], \, [R^{\vdash}_a, R^{\dashv}_b], [\widetilde{T}^{\times}_a,T^{\times}_b]\right) 
 - 2\left([L^{\dashv}_a, R^{\dashv}_b], \, [\widetilde{T}^{\times}_a, -R^{\dashv}_b], [-L^{\dashv}_a,T^{\times}_b]\right)
 \\
&\quad = \left(-[R^{\vdash}_a, R^{\dashv}_b] - 2[L^{\dashv}_a, R^{\dashv}_b], \, -[R^{\vdash}_a, R^{\dashv}_b] + 2[\widetilde{T}^{\times}_a, R^{\dashv}_b], \, -[\widetilde{T}^{\times}_a,T^{\times}_b] + 2[L^{\dashv}_a,T^{\times}_b]\right) 
\\
& \quad = (R^{\vdash}_{\langle a, b \rangle}, \, -R^{\vdash}_{\langle a, b \rangle}, \widetilde{T}^{\times}_{\langle a, b \rangle}),
\end{align*}
which finishes the proof.
\end{proof}

\begin{theorem}  \label{nalt is Malcev}
Let $D$ be a 0-dialgebra over a field of characteristic not 2 or 3. Then $(\mathrm{N_{alt}}(D), \langle -,- \rangle)$ is a Malcev dialgebra.
\end{theorem}

\begin{proof}
Taking into account Propositions \ref{nalt} and \ref{naltclosed}, the result follows from the fact that every subspace of an alternative dialgebra, which is closed under the dicommutator, is a Malcev dialgebra. (See \cite[Section 4]{BPSO} for details.)
\end{proof}

\section{On the search for a nonlinear di-Malcev identity}

In the previous section, we have shown that the generalized alternative di-nucleus of a 0-dialgebra, endowed with the dicommutator, is a Malcev dialgebra (see Theorem \ref{nalt is Malcev}). To prove Theorem \ref{nalt is Malcev}, we have used that any subspace of an alternative dialgebra, which is closed under the dicommutator, is a Malcev dialgebra. A more interesting fact would be to prove Theorem \ref{nalt is Malcev} independently; since it would give us a more general construction of Malcev dialgebras. For this task, based on the proof of the corresponding result for algebras (see \cite[Proposition 4.3]{MPI} and \cite[p. 9]{M}), one can expect that it will be crucial to have a nonlinear version of the di-Malcev identity.
 
Motivated by the following result, due to Myung, our first step will be to introduce the analogous operation to the Jacobian for dialgebras.

\begin{proposition} \cite[Proposition 1.1]{M} \label{idJacobian}
In a Malcev algebra, the Malcev identity is equivalent to the identity
\begin{equation} \label{nonlinearMalcev}
J( x, y, xz ) \equiv J( x, y, z )x,
\end{equation}
where $J( x, y, z ) = (xy)z + (yz)x + (zx)y$ is the {\bf Jacobian}.
\end{proposition}

Inspired by the fact that the Jacobian vanishes in any Lie algebra, we introduce the following trilinear operation on any Malcev dialgebra:
\[
L(x, y, z) = (xy)z - x(yz) - (xz)y.
\]
Note that $L$ equals to zero in every (right) Leibniz algebra. We will refer to $L$ as the {\bf di-Jacobian}. The following two remarks justify our terminology.

\begin{remark}
In a Malcev algebra, an application of the anticommutativity tells us that the di-Jacobian coincides with the Jacobian.
\end{remark}

\begin{remark}
The di-Jacobian could also be obtained by applying the BSO algorithm to the Jacobian. In fact, making $x$, $y$ and $z$ the center in the Jacobian gives
\begin{align*}
J_1(x, y, z) & = (x \dashv y) \dashv z + (y \vdash z) \vdash x + (z \vdash x) \dashv y,
\\
J_2(x, y, z) & = (x \vdash y) \dashv z + (y \dashv z) \dashv x + (z \vdash x) \vdash y,
\\
J_3(x, y, z) & = (x \vdash y) \vdash z + (y \vdash z) \dashv x + (z \dashv x) \vdash y.
\end{align*}
Since $J_2(x, y, z) = J_1(y, z, x)$ and $J_3(x, y, z) = J_1(z,x,y)$, we discard $J_2$ and $J_3$, and we retain $J_1$. At this point, note that in a Malcev dialgebra with product $xy$, the right product is superfluous, since $x \dashv y = - y \vdash x = xy$ (see \cite[Section 3]{BPSO} for more details). Thus, rewriting $J_1(x, y, z)$ in terms of the operation $xy$ and applying right anticommutativity, we obtain 
\begin{align*}
J_1(x, y, z) & = (x \dashv y) \dashv z + (y \vdash z) \vdash x + (z \vdash x) \dashv y
\\
& = (xy)z + x(zy) - (xz)y = (xy)z - x(yz) - (xz)y = L(x, y, z),
\end{align*}
as claimed.
\end{remark}

In what follows, we will show that the di-Jacobian satisfies some properties analogous to those of the Jacobian. We start by recalling some basic notions and facts about the so-called Malcev admissible algebras.

An algebra $A$ is called {\bf Malcev admissible} if $A^- = (A, [-,-])$ is a Malcev algebra. In the theory of Malcev admissible algebras, the trilinear operation
\begin{equation} \label{function S}
S(x,y,z):= (x,y,z) + (y,z,x) + (z,x,y),
\end{equation}
plays an important role. More precisely, given an algebra $A$, over a field $F$ of arbitrary characteristic, expanding the associators, we get
\begin{equation} \label{S and J}
S(x, y, z) - S(x, z, y) = J_{A^-}(x, y, z),
\end{equation}
where $J_{A^-}$ stands for the Jacobian of $A^-$. If $A$ is flexible, then the function $S$ alternates in its second and third arguments, i.e., $A$ satisfies $S(x, y, z) \equiv -S(x, z, y)$,
and so \eqref{S and J} applies to get 
\begin{equation} \label{idaux}
2S(x, y, z) \equiv J_{A^-}(x, y, z).
\end{equation}
From \eqref{idaux} and Proposition \ref{idJacobian} (see also \cite[Lemma 1.2 (ii)]{M}) it follows that a flexible algebra is Malcev admissible if and only if the following identity is satisfied.
\begin{equation}
2S(x, y, [x, z]) \equiv 2[S(x, y, z), x].
\end{equation}

Coming back to the dialgebra setting, our first task will be to introduce the analogue of the operation $S$. To this end, we first expand the associators in \eqref{function S}:
\begin{align} \label{S}
S(x, y, z) & = (xy)z - x(yz) + (yz)x - y(zx) + (zx)y - z(xy).
\end{align}
Next, applying the BSO algorithm (by making $x$ the center in each monomial) produces the following trilinear operation in a nonassociative dialgebra:
\begin{align*}
\widetilde{S}(x, y, z) & = (x \dashv y) \dashv z - x \dashv (y \dashv z) + (y \vdash z) \vdash x - y \vdash (z \vdash x) + (z \vdash x) \dashv y \\
& \quad - z \vdash (x \dashv y) = (x,y,z)_\dashv + (y,z,x)_\vdash + (z,x,y)_\times.
\end{align*}
Note that making $y$ or $z$ the center in \eqref{S} does not give anything new: if $S_i(x,y,z)$ is the operation obtained from $S(x,y,z)$ by making the $i$-th argument the center, then $\, \, \widetilde{S}(x, y, z) = S_1(x, y, z) = S_2(z, x, y) = S_3(y, z, x)$.
\begin{definition}
A 0-dialgebra $D$ is called {\bf Malcev admissible} if $D^- = (D, \langle -,- \rangle)$ is a Malcev dialgebra. 
\end{definition}

We have already seen an example of a Malcev admissible dialgebra, that is, the generalized alternative di-nucleus $\mathrm{N_{alt}}(D)$ of a 0-dialgebra $D$. 

Recall (see \cite[Section 7]{FFBSOK} for details) that a \textbf{flexible dialgebra} is a 0-dialgebra which satisfies the identities:
\begin{equation} \label{flexid}
(x,y,z)_\dashv + (z,y,x)_\vdash \equiv 0, \qquad (x,y,z)_\times + (z,y,x)_\times \equiv 0.
\end{equation}
Note that the first identity in \eqref{flexid} coincides with the first identity in the definition of an alternative dialgebra. Moreover, every alternative dialgebra is flexible.

The following result collects together some properties of the operation $\widetilde{S}$.

\begin{lemma} \label{aux1}
Let $D$ be a flexible dialgebra over a field of characteristic different from 2. Then 
\begin{enumerate}[(i)]\itemsep=2mm
\item[\rm(i)] $\widetilde{S}(x, y, z) = - \widetilde{S}(x, z, y)$, $\quad 2\widetilde{S}(x, y, z) = L_{D^-}(x, y, z)\,$ for all $x, y, z \in D$. 
\item[\rm(ii)] $D^-$ is a Leibniz algebra if and only if $\,\widetilde{S}(x, y, z) \equiv 0$.
\end{enumerate}
\end{lemma}

\begin{proof}
(i). For $x$, $y$, $z \in D$, it follows
\begin{align*}
\widetilde{S}(x, y, z) & = (x, y, z)_\dashv + (y, z, x)_\vdash + (z, x, y)_\times 
\\
& \stackrel{\eqref{flexid}}{\equiv} 
- (z, y, x)_\vdash - (x, z, y)_\dashv - (y, x, z)_\times = - \widetilde{S}(x, z, y).
\end{align*}
Moreover, applying the bar identities we get
\begin{align*}
L_{D^-}& (x, y, z) = \langle \langle x, y \rangle, z \rangle - \langle x, \langle y, z \rangle \rangle 
- \langle \langle x, z \rangle, y \rangle = (x \dashv y - y \vdash x) \dashv z
\\
&  - z \vdash (x \dashv y - y \vdash x) - x \dashv (y \dashv z - z \vdash y) + (y \dashv z - z \vdash y) \vdash x
\\
& - (x \dashv z - z \vdash x) \dashv y + y \vdash (x \dashv z - z \vdash x) \equiv (x, y, z)_\dashv + (y, z, x)_\vdash 
\\
& + (z, x, y)_\times - (x, z, y)_\dashv - (z, y, x)_\vdash - (y, x, z)_\times = \widetilde{S}(x, y, z) - \widetilde{S}(x, z, y) 
\\
& = 2\widetilde{S}(x, y, z), 
\end{align*}
as desired.

(ii) is a consequence of (i).
 
\end{proof}

Inspired by \eqref{nonlinearMalcev}, we introduce the following nonlinear identity in a Malcev dialgebra
\begin{equation} \label{LId}
L(y, x, zx) \equiv L(y, z, x )x.
\end{equation}
\begin{remark}
Note that in a Malcev algebra identities \eqref{nonlinearMalcev} and \eqref{LId} turn to be equal. This follows from the skew symmetries of the Jacobian and the anticommutativity law.
\end{remark}

\begin{remark}
In a Malcev admissible dialgebra, identity \eqref{LId} can be rewritten, in terms of the operation $\widetilde{S}$, as follows:
\begin{equation} \label{auxs}
\widetilde{S}(x, y, \langle z, y \rangle) \equiv \langle \widetilde{S}(x, z, y ), y\rangle.
\end{equation}
\end{remark}

At this point, it is natural ask whether identity \eqref{LId} will be the nonlinear analogue of the di-Malcev identity. Unfortunately, against what seems to be natural identity \eqref{LId} and the di-Malcev identity turn out to be non-equivalent. We will use computer algebra to prove this claim. To this end, we will regard the subspace of all identities of degree $n$ for a certain algebra $A$ as a module over the symmetric group
$S_n$ acting by permutations of the variables. Given identities $f, f_1, \dots, f_k$ of degree $n$, we say that $f$ is a consequence of $f_1, \dots, f_k$ if $f$ belongs to the $S_n$-submodule generated by $f_1, \dots, f_k$.

\begin{theorem}
The di-Malcev identity is not equivalent to identity \eqref{LId} in the free right anticommutative algebra.
\end{theorem}

\begin{proof}
A binary operation has five association types in degree 4, namely:
 \[
  ((ab)c)d, 
  \quad
  (a(bc))d, 
  \quad
  (ab)(cd),
  \quad
  a(b(cd)),
  \quad
  a((bc)d). 
  \] 
An application of the right anticommutativity law eliminates type 4, since $a(b(cd)) = - a((cd)b)$.  Moreover, types 2, 3 and 5 have the following skew-symmetries:
  \begin{equation}
  \label{skewsymmetries}
  ( a ( c b ) ) d = - ( a ( b c ) ) d, \quad
  ( a b ) ( d c ) = - ( a b ) ( c d ), \quad
  a ( ( c b ) d ) = - a ( ( b c ) d ).
  \end{equation}
Each skew-symmetry halves the number of multilinear monomials, giving the 60 monomials of 
Table \ref{FRA4basis} which form an ordered basis of the multilinear subspace 
of degree 4 in the free right anticommutative algebra on four generators.

We first process identity \eqref{LId}. We create a $48 \times 60$ matrix $M$, initialized to zero. 
We fill the first 24 rows with the coefficient vectors obtained by applying all 24 permutations of the variables $a, b, c, d$ to identity \eqref{LId} and straightening the terms by using right anticommutativity. The rank of the resulting matrix is 8. Next, we perform the same calculations with the di-Malcev identity and store 
the resulting vectors in rows $25-48$ of $M$; the rank is now 20. We then reverse this procedure, first processing the di-Malcev identity, obtaining rank 20 and then processing identity \eqref{LId}, which does not increase the rank.

We conclude that identity \eqref{LId} is a consequence of the di-Malcev identity but the converse is not true: the di-Malcev identity can not be obtained from identity \eqref{LId}. These calculations show that identity \eqref{LId} and the di-Malcev identity generate a 20-dimensional subspace in the 60-dimensional space spanned by the right anticommutative monomials. Moreover, identity \eqref{LId} generates a 8-dimensional subspace while the di-Malcev identity generates the entire 20-dimensional subspace. 

These calculations were performed by using the Maple 16 package \texttt{LinearAlgebra}.
\end{proof}
\begin{table}
  \[
  \begin{array}{llllll}
  ((ab)c)d, &\quad 
  ((ab)d)c, &\quad 
  ((ac)b)d, &\quad 
  ((ac)d)b, &\quad
  ((ad)b)c, &\quad 
  ((ad)c)b,
  \\
  ((ba)c)d, &\quad 
  ((ba)d)c, &\quad
  ((bc)a)d, &\quad 
  ((bc)d)a, &\quad 
  ((bd)a)c, &\quad 
  ((bd)c)a, 
  \\
  ((ca)b)d, &\quad 
  ((ca)d)b, &\quad 
  ((cb)a)d, &\quad 
  ((cb)d)a, &\quad
  ((cd)a)b, &\quad 
  ((cd)b)a,  
  \\
  ((da)b)c, &\quad 
  ((da)c)b, &\quad
  ((db)a)c, &\quad 
  ((db)c)a, &\quad 
  ((dc)a)b, &\quad 
  ((dc)b)a, 
  \\
  (a(bc))d, &\quad 
  (a(bd))c, &\quad 
  (a(cd))b, &\quad 
  (b(ac))d, &\quad
  (b(ad))c, &\quad 
  (b(cd))a,  
  \\
  (c(ab))d, &\quad 
  (c(ad))b, &\quad
  (c(bd))a, &\quad 
  (d(ab))c, &\quad 
  (d(ac))b, &\quad 
  (d(bc))a, 
  \\
  (ab)(cd), &\quad 
  (ac)(bd), &\quad 
  (ad)(bc), &\quad 
  (ba)(cd), &\quad
  (bc)(ad), &\quad 
  (bd)(ac),  
  \\
  (ca)(bd), &\quad 
  (cb)(ad), &\quad
  (cd)(ab), &\quad 
  (da)(bc), &\quad 
  (db)(ac), &\quad 
  (dc)(ab), 
  \\
  a((bc)d), &\quad 
  a((bd)c), &\quad 
  a((cd)b), &\quad 
  b((ac)d), &\quad
  b((ad)c), &\quad 
  b((cd)a),  
  \\
  c((ab)d), &\quad 
  c((ad)b), &\quad
  c((bd)a), &\quad 
  d((ab)c), &\quad 
  d((ac)b), &\quad 
  d((bc)a).
  \end{array}
  \]
  \caption{Right anticommutative monomials in degree 4}
  \label{FRA4basis}
  \end{table}
\begin{question}
{\rm 
The results of the present section make us to ask whether there exists a nonlinear identity, which has an expression in terms of the di-Jacobian, equivalent to the di-Malcev identity. 
}
\end{question}

\section{Conjecture: speciality on Malcev dialgebras}

P\'erez-Izquierdo and Shestakov \cite{PIS} proved that any Malcev algebra is isomorphic to a subalgebra of  
the generalized alternative nucleus $\mathrm{N_{alt}}(A)$ of a certain algebra $A$. More precisely, given a Malcev algebra $M$ they constructed an algebra $\mathrm{U}(M)$, and a monomorphism $\iota: M \to \mathrm{U}(M)^-$ such that the image of $M$ lies in the generalized alternative nucleus of $\mathrm{U}(M)$, and $\mathrm{U}(M)$ is a
universal object with respect to such homomorphisms. They showed that $\mathrm{U}(M)$ has a basis of Poincar\'e-Birkhoff-Witt type over $M$, and inherits some good properties of universal enveloping algebras of Lie algebras.

Motivated by this result and based on the results of the previous sections, it seems natural to ask whether any Malcev dialgebra arises from a subalgebra of the generalized alternative di-nucleus of a certain 0-dialgebra. We leave it as an open problem.

\section{The associative di-nucleus}

The associators vanish in any associative algebra. Concerning to nonassociative algebras, several authors have analyzed what happens if we impose that the associators satisfy certain polynomial identities. For example, Thedy \cite{T} studied the case in which all associators commute with all the elements. Later on, Kleinfeld
and Widmer \cite{KW} considered rings which associators satisfy $(x, y, z) = (y, z, x)$; previously studied by Outcalt \cite{O} and Sterling \cite{S}. In the present section, we focus our attention on a result due to Kleinfeld \cite{Kl}, which stated that a semiprime algebra with all its associators in the associative nucleus is associative. 

Let us recall that the {\bf associative nucleus} $\N(A)$ of an algebra $A$ is defined by 
\begin{equation*}
\N(A) = \{a \in A \mid (a, A, A) = (A, a, A) = (A, A, a) = 0\}.
\end{equation*}

As we have seen, the theory of dialgebras is not entirely analogous to the theory of algebras; in the sense that we can not translate directly an arbitrary result from algebras to dialgebras, and hope that the resulting result will also hold in the dialgebra setting. Likely, in this section, we will show that the analogue to Kleinfeld's theorem holds for a 0-dialgebra.

The definition of the associative di-nucleus of a 0-dialgebra can be obtained by applying the KP algorithm to the defining identities for the associative nucleus $\N(A)$ of an algebra $A$. Proceeding as in subsection \ref{Nalt} (we omit here the details) we will obtain the following definition.

\begin{definition}
Let $D$ be a 0-dialgebra. The {\bf associative di-nucleus} $\N(D)$ of $D$ is the set of elements $a \in D$, which satisfies the bar identities jointly with the following identities.
\allowdisplaybreaks
\begin{alignat}{3}
(a, D, D)_\dashv & \equiv (D, a, D)_\dashv & \equiv (D, D, a)_\dashv & \equiv 0, 
\tag{AN1} \label{AN1}
\\
(a, D, D)_\times & \equiv (D, a, D)_\times & \equiv (D, D, a)_\times & \equiv 0,
\tag{AN2} \label{AN2}
\\
(a, D, D)_\vdash & \equiv (D, a, D)_\vdash & \equiv (D, D, a)_\vdash & \equiv 0.
\tag{AN3} \label{AN3}
\end{alignat}
\end{definition}

One of the differences between the generalized alternative di-nucleus $\mathrm{N_{alt}}(D)$ and the associative nucleus $\N(D)$ is that $\N(D)$ is a subdialgebra of $D$. 

In order to prove this important property of $\N(D)$, we need to introduce some notions. The following identity, so-called the {\bf Teichm\"uller identity}
\allowdisplaybreaks
\begin{equation} \label{T} \tag{T}
(wx, y, z) - (w, xy, z) + (w, x, yz) \equiv w(x, y, z) + (w, x, y)z,
\end{equation}
holds in any algebra. Due to the relation between the BSO and the KP algorithms the following identities hold in any 0-dialgebra. We will refer to them as the {\bf Teichm\"uller di-identities}. 
\allowdisplaybreaks
\begin{align}
(w \dashv x, y, z)_\dashv - (w, x \dashv y, z)_\dashv + (w, x, y \dashv z)_\dashv  & \equiv w \dashv (x, y, z)_\dashv + (w, x, y)_\dashv \dashv z, \tag{T1} \label{T1}
\\
(w \vdash x, y, z)_\dashv - (w, x \dashv y, z)_\times + (w, x, y \dashv z)_\times  & \equiv w \vdash (x, y, z)_\dashv + (w, x, y)_\times \dashv z, \tag{T2} \label{T2}
\\
(w \vdash x, y, z)_\times - (w, x \vdash y, z)_\times + (w, x, y \dashv z)_\vdash  & \equiv w \vdash (x, y, z)_\times + (w, x, y)_\vdash \dashv z, \tag{T3} \label{T3}
\\
(w \vdash x, y, z)_\vdash - (w, x \vdash y, z)_\vdash + (w, x, y \vdash z)_\vdash  & \equiv w \vdash (x, y, z)_\vdash + (w, x, y)_\vdash \vdash z. \tag{T4} \label{T4}
\end{align}
Note that \eqref{T1}-\eqref{T4} are obtained by expanding the associators in \eqref{T}, and making $w$, $x$, $y$ and $z$, respectively, the center of each monomial.

\begin{lemma}
The associative di-nucleus $\N(D)$ of a 0-dialgebra $D$ is a subdialgebra.
\end{lemma}

\begin{proof}
We will show that $\N(D)$ is closed under the left product. Similarly, one can prove that it is also closed under the right product.

Given $a, b \in \N(D)$ applying \eqref{T1}, by taking into account that \eqref{AN1} holds for $a$ and $b$, we get that $a \dashv b$ satisfies \eqref{AN1}. Applications of the left bar identity and \eqref{T4} give that \eqref{AN3} is satisfied by $a \dashv b$. To finish, in order to show that $a \dashv b$ also verifies \eqref{AN2} apply the left bar identity jointly with \eqref{T2} and \eqref{T3}.
\end{proof}

Any nonassociative algebra has a particular ideal, called the {\bf associator ideal} defined to be the smallest ideal which contains all associators. Kleinfeld \cite{Kl} noticed that its elements are either finite sums of associators or right multiples of associators. In what follows, we will develop the necessary machinery to find a similar notion in the dialgebra setting. 

\begin{definition}
A subspace $I$ of a dialgebra $D$ is called a {\bf di-ideal} if it satisfies that $I\dashv D$, $I\vdash D$, $D\dashv I$, $D \vdash I \subseteq I$. 
\end{definition}

Let $D$ be a 0-dialgebra, let us denote by $\Ass(D)$ the set consisting of all finite sums of dialgebra associators of $D$ jointly with all its right multiples of dialgebra associators of $D$. To be more precise, an arbitrary element of $\Ass(D)$ is of one of the following types:
\begin{itemize}
\item A finite sum of dialgebra associators: 
\allowdisplaybreaks
\[
(x, y, z)_\dashv, \quad (x, y, z)_\times, \quad (x, y, z)_\vdash 
\] 
\item A right multiple of a dialgebra associator:
\allowdisplaybreaks 
\[
(x, y, z)_\dashv \dashv t, \quad (x, y, z)_\times \dashv t, \quad (x, y, z)_\vdash \dashv t, \quad (x, y, z)_\vdash \vdash t
\]
\end{itemize}
where $x$, $y$, $z$, $t \in D$.
\begin{remark}
Notice that the bar identities apply to get 
\allowdisplaybreaks
\[
(x, y, z)_\dashv \vdash t \equiv (x, y, z)_\times \vdash t \equiv (x, y, z)_\vdash \vdash t.
\]
\end{remark}

\begin{lemma}
Let $D$ be a 0-dialgebra. Then $\Ass(D)$ is a di-ideal of $D$. Moreover, $\Ass(D)$ is the smallest di-ideal of $D$ containing all the dialgebra associators.
\end{lemma}

\begin{proof}
Due to the bar identities, the result follows by noticing the following:
\allowdisplaybreaks
\begin{align*}
& \left( (x, y, z)_\star \dashv t \right) \dashv u  = 
\left( (x, y, z)_\star, t, u \right)_\dashv - (x, y, z)_\star \dashv (t \dashv u), 
\\
& \left( (x, y, z)_\vdash \vdash t \right) \vdash u = 
\left( (x, y, z)_\vdash, t, u \right)_\vdash - (x, y, z)_\vdash \vdash (t \vdash u),
\\
& u \dashv (x, y, z)_\dashv  \stackrel{\eqref{T1}}{\equiv} (u \dashv x, y, z)_\dashv 
- (u, x \dashv y, z)_\dashv + (u, x, y \dashv z)_\dashv - (u, x, y)_\dashv \dashv z,
\\
& u \vdash (x, y, z)_\dashv  \stackrel{\eqref{T2}}{\equiv} (u \vdash x, y, z)_\dashv 
- (u, x \dashv y, z)_\times + (u, x, y \dashv z)_\times - (u, x, y)_\times \dashv z,
\\
& u \vdash (x, y, z)_\times  \stackrel{\eqref{T3}}{\equiv} (u \vdash x, y, z)_\times
- (u, x \vdash y, z)_\times + (u, x, y \dashv z)_\vdash - (u, x, y)_\vdash \dashv z,
\\
& u \vdash (x, y, z)_\vdash  \stackrel{\eqref{T4}}{\equiv} (u \vdash x, y, z)_\vdash
- (u, x \vdash y, z)_\vdash + (u, x, y \vdash z)_\vdash - (u, x, y)_\vdash \vdash z,
\\
& u \vdash \left( (x, y, z)_\star \dashv t \right) = (u, (x, y, z)_\star, t)_\times - \left(u \vdash (x, y, z)_\star \right) \dashv t,  
\\
& u \vdash \left( (x, y, z)_\vdash \vdash t \right) = (u, (x, y, z)_\vdash, t)_\vdash - \left(u \vdash (x, y, z)_\vdash \right) \vdash t,
\\
& u \dashv \left( (x, y, z)_\dashv \dashv t \right) = (u, (x, y, z)_\dashv, t)_\dashv - \left(u \dashv (x, y, z)_\dashv \right) \dashv t,
\end{align*}
where $\star \in \{\dashv, \, \times, \, \vdash \}$.
\end{proof}

\begin{definition}
The {\bf associator di-ideal} of a 0-dialgebra $D$ is the di-ideal $\Ass(D)$.
\end{definition}

\begin{definition}
We say that a dialgebra is {\bf semiprime} if it does not contain di-ideals which left and right square products are zero; that is, $I \dashv I = I \dashv I = 0$ implies $I = 0$ for every di-ideal $I$.
\end{definition}

\begin{theorem} \label{main2}
Let $D$ be a 0-dialgebra over a field of characteristic not 2. Assume that $D$ is semiprime and has all its associators in $\N(D)$. Then $D$ is an associative dialgebra.
\end{theorem}

\begin{lemma}
Let $D$ a 0-dialgebra. Then for $a \in \N(D)$, $x, y, z \in D$ and $\star \in \{\dashv, \, \times, \, \vdash \}$, the following identities hold:
\allowdisplaybreaks
\begin{align}
& (a \dashv x, y, z)_\dashv \equiv a \dashv (x, y, z)_\dashv,  \label{uno1}
\\
& (a \vdash x, y, z)_\star \equiv a \vdash (x, y, z)_\star,
\\
& (x \dashv a, y, z)_\star \equiv (x, a \dashv y, z)_\star, \label{uno2}
\\
& (x \vdash a, y, z)_\dashv \equiv (x, a \dashv y, z)_\times,
\\
& (x, y \dashv a, z)_\star \equiv (x, y, a \vdash z)_\star, \label{uno3}
\\
& (x, y \vdash a, z)_\times \equiv (x, y, a \dashv z)_\vdash,
\\
& (x, y, z \dashv a)_\star \equiv (x, y, z)_\star \dashv a, \label{uno4}
\\
& (x, y, z \vdash a)_\vdash \equiv (x, y, z)_\vdash \vdash a.
\end{align}
\end{lemma}

\begin{proof}
It follows by applying the Teichm\"uller di-identities and the bar identities.
\end{proof}

\smallskip

\noindent \emph{Proof of Theorem \ref{main2}.} Write $I$ to denote the associative di-nucleus $\Ass(D)$ of $D$. We are going to show that $I\dashv I = I \vdash I = 0$ which, by the semiprimeness of $D$, will allow us to conclude that $I = 0$. 

Given $x, y, z, t, u, v \in D$ and $\star \in \{\dashv, \, \times, \, \vdash \}$ we claim that 
\allowdisplaybreaks
\begin{alignat}{3}
(x, y, z)_\star \dashv (t, u, v)_\dashv & \equiv (x, y, z)_\star \dashv (t, u, v)_\times & \equiv (x, y, z)_\star \dashv (t, u, v)_\vdash = 0, \label{p01}
\\
(x, y, z)_\dashv \vdash (t, u, v)_\star & \equiv (x, y, z)_\times \vdash (t, u, v)_\star & \equiv (x, y, z)_\vdash \vdash (t, u, v)_\star = 0, \label{p02}
\end{alignat}

Set $p := (x, y, z)_\dashv \dashv (t, u, v)_\dashv$. Since $(x, y, z)_\dashv \in \N(D)$, an application of \eqref{uno1}
gives $p \equiv ((x, y, z)_\dashv \dashv t, u, v)_\dashv$. On the other hand from \eqref{T1} we obtain
\allowdisplaybreaks
\[
(x, y, z)_\dashv \dashv t \equiv (x \dashv y, z, t)_\dashv - (x, y \dashv z, t)_\dashv 
+ (x, y, z \dashv t)_\dashv - x \dashv (y, z, t)_\dashv, 
\]
which by the hypothesis yields $p \equiv - (x \dashv (y, z, t)_\dashv, u, v)_\dashv$. Next \eqref{uno2} and \eqref{T1}
apply to get 
$p \equiv -(x,(y, z, t)_\dashv \dashv u, v)_\dashv \equiv (x, y \dashv (z, t, u)_\dashv v)_\dashv $. Using the 
right bar identity, \eqref{uno3} and \eqref{T1} we have 
$p \equiv (x, y, (z, t, u)_\dashv \dashv v)_\dashv \equiv -(x, y, z \dashv (t, u, v)_\dashv)_\dashv$. 
To finish apply \eqref{uno4} to get $p \equiv -(x, y, z)_\dashv \dashv (t, u, v)_\dashv = -p$. Thus $2p \equiv 0$ and therefore $p \equiv 0$. Reasoning in a similar way, one can complete the proof of \eqref{p01} and show \eqref{p02}.

Next, for $x, y, z, s, t, u, v, w \in D$ and $\star \in \{\dashv, \, \times, \, \vdash \}$ we claim that
\allowdisplaybreaks
\begin{alignat}{5}
& ((x, y, z)_\star & \dashv t) & \dashv ((u, v, w)_\dashv & \dashv s) \equiv 
  ((x, y, z)_\star & \dashv t) & \dashv ((u, v, w)_\times & \dashv s) \equiv
\\ \notag
& ((x, y, z)_\star & \dashv t) & \dashv ((u, v, w)_\vdash & \dashv s) \equiv
  ((x, y, z)_\star & \dashv t) & \dashv ((u, v, w)_\vdash & \vdash s) = 0,
\\   
& ((x, y, z)_\dashv & \dashv t) & \vdash ((u, v, w)_\star & \dashv s) \equiv 
  ((x, y, z)_\times & \dashv t) & \vdash ((u, v, w)_\star & \dashv s) \equiv
\\ \notag
& ((x, y, z)_\vdash & \dashv t) & \vdash ((u, v, w)_\star & \dashv s) \equiv
  ((x, y, z)_\vdash & \vdash t) & \vdash ((u, v, w)_\star & \dashv s) \equiv 0,
\\ 
& ((x, y, z)_\vdash & \vdash t) & \dashv ((u, v, w)_\dashv & \dashv s) \equiv 
  ((x, y, z)_\vdash & \vdash t) & \dashv ((u, v, w)_\times & \dashv s) \equiv
\\ \notag
& ((x, y, z)_\vdash & \vdash t) & \dashv ((u, v, w)_\vdash & \dashv s) \equiv
  ((x, y, z)_\vdash & \vdash t) & \dashv ((u, v, w)_\vdash & \vdash s) \equiv 0,
\\ \label{last}
& ((x, y, z)_\vdash & \vdash t) & \vdash ((u, v, w)_\vdash & \vdash s) \equiv 
  ((x, y, z)_\dashv & \dashv t) & \vdash ((u, v, w)_\vdash & \vdash s) \equiv
\\ \notag
& ((x, y, z)_\times & \dashv t) & \vdash ((u, v, w)_\vdash & \vdash s) \equiv
  ((x, y, z)_\vdash & \dashv t) & \vdash ((u, v, w)_\vdash & \vdash s) \equiv 0.
\end{alignat}
Let us check that $((x, y, z)_\vdash \vdash t) \vdash ((u, v, w)_\vdash \vdash s) \equiv 0$. Similarly, one can 
show that all the elements above equal zero.

Notice that $((x, y, z)_\vdash \vdash t) \vdash (u, v, w)_\vdash \equiv (x, y, z)_\vdash \vdash (t \vdash (u, v, w)_\vdash)$, 
since $((x, y, z)_\vdash, t, (u, v, w)_\vdash)_\vdash \equiv 0$ by the hypothesis and \eqref{AN3}. Thus, it makes sense to write $(x, y, z)_\vdash \vdash t \vdash (u, v, w)_\vdash$. From \eqref{T4} we get
\[
(x, y, z)_\vdash \vdash t \equiv (x \vdash y, z, t)_\vdash - (x, y \vdash z, t)_\vdash + (x, y, z \vdash t)_\vdash 
- x \vdash( y, z, t)_\vdash,
\] 
which yields
\allowdisplaybreaks
\begin{align*}
(x, y, z)_\vdash \vdash t \vdash (u, v, w)_\vdash & \equiv 
  (x \vdash y, z, t)_\vdash \vdash (u, v, w)_\vdash
- (x, y \vdash z, t)_\vdash \vdash (u, v, w)_\vdash +
\\
& \quad \, \, (x, y, z \vdash t)_\vdash \vdash (u, v, w)_\vdash
- (x \vdash( y, z, t)_\vdash) \vdash (u, v, w)_\vdash
\\
& \equiv 
- x \vdash( y, z, t)_\vdash \vdash (u, v, w)_\vdash \equiv 0,
\end{align*}
by \eqref{AN3} and \eqref{p02}. Then applying \eqref{AN3} we get
\allowdisplaybreaks
\[
((x, y, z)_\vdash \vdash t) \vdash ((u, v, w)_\vdash \vdash s) = 
((x, y, z)_\vdash \vdash t \vdash (u, v, w)_\vdash) \vdash s \equiv 0,
\]
as desired. To finish, notice that \eqref{p01}-\eqref{last} yield that $I\dashv I = I \vdash I = 0$, which concludes the proof.


\section*{Acknowledgements}

The author thanks Joe Repka for his carefully reading of this manuscript. She also thanks 
Sara Madariaga for her help with the Maple calculations. 
The author was supported by the Spanish MEC and Fondos FEDER 
jointly through project MTM2010-15223, and by the Junta de Andaluc\'ia (projects FQM-336 and FQM2467).

\end{document}